\documentclass[a4paper, 10pt, conference]{ieeeconf}      

\IEEEoverridecommandlockouts                              

\usepackage{xcolor}
\usepackage{amsmath}
\usepackage{amssymb}
\usepackage{mathtools}
\usepackage{graphicx}
\usepackage{enumerate}
\newtheorem{definition}{Definition}
\newtheorem{assumption}{Assumption}
\newtheorem{theorem}{Theorem}
\newtheorem{remark}{Remark}
\newtheorem{proposition}{Proposition}

\newtheorem{lemma}{Lemma}

\title{\LARGE \bf
Output-Lifted Learning Model Predictive Control
}

\author{Siddharth H. Nair, Ugo Rosolia,  Francesco Borrelli
\thanks{SHN, FB are with the Department of Mechanical Engineering, University of California Berkeley. Email IDs: \{siddharth\_nair, fborrelli\}@berkeley.edu}
\thanks{UR is with the Department of Mechanical and Civil Engineering, California Institute of Technology. Email ID: urosolia@caltech.edu}
}

\begin{document}

\maketitle
\thispagestyle{empty}
\pagestyle{empty}

\begin{abstract}
We propose a computationally efficient Learning Model Predictive Control (LMPC) scheme for constrained optimal control of a class of nonlinear systems where  the  state and input can be reconstructed using lifted outputs.
For the considered class of systems, we show how to use historical trajectory data collected during iterative tasks to construct a convex value function approximation along with a convex safe set in a lifted space of virtual outputs. These constructions are iteratively updated with historical data and used to synthesize predictive control policies. We show that the proposed strategy guarantees recursive constraint satisfaction, asymptotic stability and non-decreasing closed-loop performance at each policy update. Finally, simulation results demonstrate the effectiveness of the proposed strategy on a piecewise affine (PWA) system, kinematic unicycle and bilinear DC motor.

\end{abstract}
\section{Introduction}

Infinite-horizon optimal control has a long and celebrated history, with the cornerstones laid in the 1950s by  \cite{pontryagin2018mathematical} and \cite{bellman1966dynamic}. The problem involves seeking a control signal that minimizes the cost incurred by a trajectory of a dynamical system starting from an initial condition over an infinite time horizon. While certain problem settings admit analytical solutions (like unconstrained LQR (\cite{kwakernaak1972linear})), the infinite-horizon optimal control problem for general nonlinear dynamical systems subject to constraints, is challenging to solve. This is because these problems require the numerical solution of an infinite-dimensional optimization problem, which is intractable even in the discrete-time setting (where the solution is an infinite sequence of control inputs instead of a control input signal).

Model Predictive Control (MPC) 
is an attractive methodology for tractable synthesis of feedback control of constrained nonlinear discrete-time systems. The control action at every instant requires the solution of a finite-horizon optimal control problem with a suitable constraint and cost on the terminal state of the system to approximate the infinite-horizon problem. These terminal components are designed so that the closed-loop system is stable and satisfies constraints. This is achieved by constraining the terminal state to lie in a control invariant set with an associated Control Lyapunov function (CLF). 
The computation of these sets with an accompanying CLF for nonlinear systems is challenging, in general, and a  proper review goes outside the scope of this article.


For iterative tasks where the system starts from the same position for every iteration of the optimal control problem, data from previous iterations may be used to update the MPC design using ideas from Iterative Learning Control (ILC)~(\cite{ILCsurvey,ILC_MPC_cost, ILC_MPC_model}).
In these strategies the goal of the controller is to track a given reference trajectory, and the tracking error from the previous execution is used to update the controller. For control problems where a reference trajectory may be hard to compute, \cite{UgoTAC} proposed a
reference-free iterative policy synthesis strategy, called Learning Model Predictive Control (LMPC) which iteratively constructs a control invariant terminal set and an accompanying terminal cost function using historical data. These quantities are discrete, therefore the LMPC relies on the solution of a Mixed-Integer Nonlinear Program (MINLP) at each instant for guaranteed stability and constraint satisfaction. 
We build on the work of~\cite{UgoTAC} and propose a strategy to reduce the computational burden for a class of nonlinear systems by replacing these discrete sets and functions with convex ones while still maintaining safety and performance  guarantees.

In this work, we present a LMPC framework for a class of discrete-time nonlinear systems for which the state and input can be reconstructed using \textit{lifted} outputs. These \textit{lifted} outputs are constructed using \textit{flat} outputs (\cite{flat_def}) which have also been used in \cite{aranda1996linearization} to construct dynamic feedback linearizing inputs for discrete-time systems. Existing works on constrained control for such systems require a carefully designed reference trajectory which is then tracked using MPC with a linear model obtained either by a first order approximation (\cite{de2009flatness}) or by feedback linearization (\cite{wang2019flatness,greeff2018flatness, kandler2012differential}). In both cases, there are no formal guarantees of closed-loop system stability and constraint satisfaction.  The contribution of this article is twofold. First, we show how to construct convex terminal set and terminal cost using historical \textit{lifted} output data for the MPC optimization problem.  Second, we show that with some mild assumptions, a convex synthesis of the terminal cost is permissible on the space of \textit{lifted} outputs. As opposed to the discrete formulation of the terminal set and cost in \cite{UgoTAC} (thus requiring solutions to MINLPs), our formulation enables us to solve continuous Nonlinear Programs (NLPs). This can significantly decrease the computational overheads associated with computing the control action at each instant. Second, we show that the revised LMPC strategy in closed-loop ensures $i$) constraint satisfaction $ii$) convergence to equilibrium $iii$) non-decreasing closed-loop system performance across iterations.  

The paper is organized as follows. We begin by formally describing the problem we want to solve in Section~\ref{sec:PD} along with required definitions. Section~\ref{sec:LMPCF} shows how to construct the terminal set and terminal cost in the lifted output space and it introduces the control design. The closed-loop properties are analysed in Section~\ref{sec:LMPC_analysis}. Finally, Section~\ref{sec:ex} presents numerical results that illustrate our proposed approach on three examples.
\section{Problem Formulation}\label{sec:PD}
Consider a nonlinear discrete-time system given by the dynamics
\begin{align}\label{eq:sysdyn}
    x_{t+1}=f(x_t,u_t),
\end{align}
 where $x_t\in\mathbb{X}\subseteq \mathbb{R}^n$ and $u_t\in\mathbb{U}\subseteq \mathbb{R}^m$  are the system state and input respectively at time $t$. Let $x_F$ be an unforced equilibrium of \eqref{eq:sysdyn}, $x_F=f(x_F,0)$ with $f(\cdot,0)$ being continuous at $x_F$.
The \textit{lifted} output for the nonlinear system \eqref{eq:sysdyn} is defined below.
\begin{definition}\label{def:diffFlat} Let $y_t=h(x_t)$ with $h : \mathbb{X}\rightarrow \mathbb{R}^m$ be the output of system \eqref{eq:sysdyn}. If $\exists R\in\mathbb{N}$ and a function $\mathcal{F}:\mathbb{R}^{m\times R+1}\rightarrow \mathbb{X}\times\mathbb{U}$,
such that any the state/input pair ($x_t$, $u_t$) can be uniquely reconstructed from a sequence of outputs $y_t,\dots, y_{t+R}$
\begin{align}\label{eq:flatdef}
    (x_{t},u_{t})&=\mathcal{F}([y_{t},y_{t+1},\dots,y_{t+R}]),
\end{align}
then the lifted output is the matrix \begin{align}\label{eq:vrtl_lftd_op}
\mathbf{Y}_t=[y_t,\dots,y_{t+R}]\in\mathbb{R}^{m\times R+1}.\end{align}
\end{definition}
\begin{remark} The output $y_t=h(x_t)$ corresponding to the lifted output in definition \ref{def:diffFlat} is also called a flat output in \cite{flat_def}.
For linear discrete-time systems, the existence of the lifted output is equivalent to the system being controllable and \textit{strongly} observable with the output $y_t=Cx_t$ (\cite{yong2015computational}).
\end{remark}
\begin{assumption}\label{ass:flat}
Given $y_t=h(x_t)$, system \eqref{eq:sysdyn} admits a lifted output $\mathbf{Y}_t$ as defined in (\ref{eq:vrtl_lftd_op}) for some finite $R$.
\end{assumption}

Consider the following infinite-horizon constrained optimal control problem for system \eqref{eq:sysdyn} with initial state $x_0=x_S$,\small
\begin{equation}\label{eq:OP_inf}
	\begin{aligned}
	J^*_{0\rightarrow\infty}(x_S)=\min_{u_0,u_1,\ldots}\quad & \displaystyle\sum\limits_{k\geq0} c(x_k,u_k) \\[1ex]
		\text{s.t.} \quad & x_{k+1}=f(x_k,u_k), ~\forall k\geq0\\
    & x_k\in\mathcal{X}, u_k\in\mathcal{U},~\forall k\geq0\\
	& x_0 = x_S.
	\end{aligned}
\end{equation}\normalsize
The state constraints $\mathcal{X}$ and input constraints $\mathcal{U}$ are described by convex sets, and $c(\cdot,\cdot)$ is a continuous, convex and positive definite function that is $0$ only at the equilibrium, $c(x_F,0)=0$. Observe that due to continuity and positive definiteness of stage cost $c(\cdot)$, a trajectory corresponding to the optimizer of \eqref{eq:OP_inf} must necessarily have its state converge to $x_F$.

We aim to synthesize a state-feedback policy that approximates the solution to the infinite-horizon (and infinite-dimensional) problem \eqref{eq:OP_inf} such that it captures its most desirable properties: $(i)$ constraint satisfaction (feasibility) and $(ii)$ asymptotic convergence (bounded cost). To tackle the infinite-dimensional nature of the problem, we use MPC which solves finite-horizon versions of \eqref{eq:OP_inf} at each time step. To ensure that the MPC has the desired properties, we build on the Learning Model Predictive Control (LMPC) framework which solves problem \eqref{eq:OP_inf} iteratively using historical data. In the next section, we proceed to briefly describe these two techniques.

\begin{remark}
Notice that to streamline the presentation we considered iterative tasks. However, the proposed strategy can be used also when the initial condition changes at each iteration. As we will show in Section~V, to guarantee safety and closed-loop stability is it only required that the initial condition belongs to the region of attraction of the LMPC~policy. 
\end{remark}

\begin{remark}
We considered deterministic systems, but the proposed strategy may be extended to handle additive uncertainty using robust tube MPC strategies~\cite{mayne2011tube}. The key idea is to leverage the proposed approach to control the nominal dynamics and use a precomputed feedback gain to handle the uncertainty, as discussed in~\cite[Section~7]{mtlmpc}.
\end{remark}

\section{Preliminaries}
\subsection{Model Predictive Control}
Consider the following finite-horizon problem at each time $t$ from state $x_t$.\small
\begin{equation}\label{eq:OP_MPC}
	\begin{aligned}
	J_{t\rightarrow t+N}(x_t)= \min\limits_{\mathbf{u_t}} \quad & \displaystyle Q(x_{N|t})+\sum\limits_{k=0}^{N-1} c(x_{k|t},u_{k|t}) \\[1ex]
		\text{s.t.}\quad  & x_{k+1|t}=f(x_{k|t},u_{k|t}), \\
    & x_{k|t}\in\mathcal{X}, u_{k|t}\in\mathcal{U},\\
    & x_{N|t}\in\mathcal{X}_f,\\
	&  x_{0|t} = x_t\\
	&\forall k\in \{0,\dots, N-1\}
	\end{aligned}
\end{equation}\normalsize
where $\mathbf{u_t}=[u_{0|t},\ldots,u_{N-1|t}]$, the initial condition $x_0=x_S$, $\mathcal{X}_f\subseteq \mathcal{X}$ is a control invariant set \cite{mpc} for the system \eqref{eq:sysdyn} with associated Control Lyapunov Function (CLF) \cite{mayne2000constrained} $Q(\cdot)$ for the equilibrium $x_F$ chosen as the terminal cost function. If $\mathbf{u}^*_t=[u^*_{0|t},\dots, u^*_{N-1|t}]$ is the minimizer of \eqref{eq:OP_MPC}, then the MPC controller is given by 
\begin{align}\label{eq:MPC}
    u_t=\pi_{MPC}(x_t)=u^*_{0|t}
\end{align}
The control invariant set $\mathcal{X}_f$ and the CLF $Q(\cdot)$ are coupled with each other and are critical to ensuring that the MPC policy \eqref{eq:MPC} in closed-loop yields a feasible and stabilizing solution to the infinite-horizon problem \eqref{eq:OP_inf}. Observe that if the optimal cost $J^*_{t\rightarrow\infty}(x_t)=\sum_{k\geq t} c(x^*_k,u^*_k)$ was known $\forall t\geq 0$, setting $Q(x_{N|t})=J^*_{t+N\rightarrow\infty}(x_{N|t})$ solves \eqref{eq:OP_inf} without requiring a terminal constraint $\mathcal{X}_f$ in \eqref{eq:OP_MPC}. The other extreme case is setting $\mathcal{X}_f=\{x_F\}$ in \eqref{eq:OP_MPC} which would yield a stable and feasible solution without requiring a terminal cost $Q(\cdot)$. That being said, computing $J^*_{t\rightarrow\infty}(\cdot)$ exactly is possible only in trivial  cases and setting $\mathcal{X}_f=\{x_F\}$ may lead to an infeasible optimization problem if $x_F$ is not reachable from $x_S$ in $N$ steps. The goal is to design $\mathcal{X}_f$ and $Q(\cdot)$ so that \eqref{eq:OP_MPC} is feasible for all $t\geq 0$ while capturing the convergence properties of the infinite-horizon optimal control problem. 
\subsection{Learning Model Predictive Control}
LMPC iteratively approximates the solution of \eqref{eq:OP_inf} using the MPC problem \eqref{eq:OP_MPC}. At iteration $j$, it uses historical data in the form of state-input trajectories from completed iterations $i\in\{0,1,\dots\,j-1\}$ to construct the terminal set $\mathcal{X}_f$ and terminal cost $Q(\cdot)$. Let $x_t^j$, $u_t^j$ and $y_t^j$ be the state, input and output of the system respectively at time $t$, corresponding to the $j$th iteration. At iteration $j$, the terminal set $\mathcal{X}^j_f$ and terminal cost $Q^{j-1}(\cdot)$ are defined as follows:
\begin{align}\label{eq:SS_def}\small
    \mathcal{X}^j_f&=\mathcal{SS}^{j-1}=\bigcup\limits_{i=0}^{j-1}\bigcup\limits_{t\geq0}\{x^i_t\}\\
    Q^{j-1}(x)&= \begin{cases}
    \min\limits_{(i,t)\in\mathcal{I}_{j-1}(x)}\sum\limits_{k\geq t}c(x^i_k,u^i_k) & x\in\mathcal{SS}^{j-1}\\
    \infty & x\not\in\mathcal{SS}^{j-1}
    \end{cases}
\end{align}\normalsize
where $\mathcal{I}_{j-1}(x)=\{(i,t)| x=x^i_t\in \mathcal{SS}^{j-1}\}$. Simply stated, the terminal set is chosen as the collection of states from previous iterations (the safe set $\mathcal{SS}^{j-1}$) and the terminal cost ($Q^{j-1}(\cdot)$) at these states is the cost of the trajectory obtained starting from that state. The terminal set is discrete and the terminal cost function is only defined on these discrete states which makes \eqref{eq:OP_MPC} a mixed-integer nonlinear program (MINLP). The computational overhead for computing such MINLP solutions is prohibitive for online, repeated solutions of \eqref{eq:OP_MPC}.

We would like to investigate if the lifted output in definition \ref{def:diffFlat} helps alleviate the combinatorial nature of the optimization problem for more tractable synthesis of feedback control to solve problem \eqref{eq:OP_inf}.

\section{Output-Lifted LMPC}\label{sec:LMPCF}
In this section we present our LMPC design using flat outputs. First, we highlight some technical assumptions that we impose on the lifted output map $\mathcal{F}(\cdot)$ from equation \eqref{eq:flatdef}. We then show how to use the stored lifted outputs from previous iterations to construct a \textit{convex} safe set in the lifted output space. The constructed set is shown to be control invariant in this space and therefore it can be used to guarantee safety in a receding horizon scheme. Afterwards, we construct a \textit{convex} terminal cost in the lifted output space and prove that it is a CLF on the constructed set. Finally, we combine these components and present our control design.

\subsection{lifted Output Map Properties}\label{sec:flat}
In addition to the existence of the lifted outputs in definition~\ref{def:diffFlat}, we require that the map $\mathcal{F}(\cdot)$ has the properties described in the following assumption. 
\begin{assumption}\label{ass:flat_prop}
The lifted output $\mathbf{Y}_t$ corresponding to $y_t=h(x_t)$ and the map $\mathcal{F}(\cdot)$ satisfy the following properties:
\begin{enumerate}[\label=(A)]
    \item \label{ass:flat_class}The map $\mathcal{F}(\cdot)$ in \eqref{eq:flatdef} requires $R$ and $R+1$ outputs for identifying the state and the input, respectively. More formally, we have that\small
\begin{align}
    x_t&=\mathcal{F}_x([y_t,y_{t+1},\dots,y_{t+R-1}])\label{eq:flatclass_x}\\
    u_t&=\mathcal{F}_u([y_t,y_{t+1},\dots,y_{t+R}])\label{eq:flatclass_u}
\end{align}\normalsize
\item The map \label{ass:flat_cont} $\mathcal{F}=(\mathcal{F}_x,\mathcal{F}_u):\mathbb{R}^{m\times R+1}\rightarrow \mathbb{X}\times\mathbb{U}$ is continuous at $\mathbf{Y}_F=[y_F,\dots,y_F]\in\mathbb{R}^{m\times R+1}$ where $y_F=h(x_F)$.
\item\label{ass:flatmap_convex}Let $\mathcal{F}^i: \mathbb{R}^{m\times R+1}\rightarrow \mathbb{R}$ be the $i$th component of the map $\mathcal{F}:\mathbb{R}^{m\times R+1}\rightarrow \mathbb{X}\times\mathbb{U}\subset\mathbb{R}^{n+m}$ where $i=1,\dots, n+m$. Then $\forall i\in\{1,\dots, n+m\}$, the maps $\mathcal{F}^i$ are monotonic on any line restriction, i.e., $\mathcal{F}^i(t\mathbf{y}_1 + (1-t)\mathbf{y}_2)$ is monotonic $\forall \mathbf{y}_1,\mathbf{y}_2\in\mathbb{R}^{m\times R+1}$, $t\in[0,1]$.
\end{enumerate}
\end{assumption}

The intuition for consideration of outputs in Assumptions~\ref{ass:flat_prop}\eqref{ass:flat_class} arises from observing the kinematics of simple mechanical systems, where the kinematics are not affected explicitly by control. 
Assumption \ref{ass:flat_prop}\eqref{ass:flat_cont} is technical, and it is required for showing that an optimizer of the infinite-horizon optimal control problem stabilizes the system to $x_F$. The following proposition clarifies the need of Assumption \ref{ass:flat_prop}\eqref{ass:flatmap_convex} for constraint satisfaction in the LMPC.

    
    



\begin{proposition}\label{prop:box}
Suppose that $\mathcal{X}\subset\mathbb{R}^n$ and $\mathcal{U}\subset\mathbb{R}^m$ are given by box constraints, $\Vert D_xx-d_x\Vert_\infty\leq1$ and $\Vert D_uu-d_u\Vert_\infty\leq1$ respectively for some real, constant diagonal matrices $D_x, D_u$ and vectors $d_x, d_u$. Let $\{\mathbf{Y}^1,\dots,\mathbf{Y}^p\}$ be any set of lifted outputs such that $\mathcal{F}(\mathbf{Y}^j)\in\mathcal{X}\times\mathcal{U}$ for each $j=1,\dots, p$. Then if assumption \ref{ass:flat_prop}\eqref{ass:flatmap_convex} holds, we have $\mathcal{F}(\mathbf{y})\in \mathcal{X} \times \mathcal{U}$ for any $\mathbf{Y}\in\textrm{conv}(\{\mathbf{Y}^1,\dots,\mathbf{Y}^p\})$.
\end{proposition}
\begin{proof} See Appendix, \ref{proof:box}.\end{proof}
Assumption \ref{ass:flat_prop}\eqref{ass:flatmap_convex} is equivalent to requiring that each $\mathcal{F}^i(\cdot)$ is quasiconvex and quasiconcave (\cite[Section 3.4]{boyd}). In Sections~\ref{eg:DC} and~\ref{eg:uni}, we see that either may be relaxed depending on the domain of dynamics $f(\cdot,\cdot)$, which could trivially lower or upper bound the state space or input space. In view of proposition \ref{prop:box}, we make the following assumption.
\begin{assumption}\label{ass:box}
The state constraints $\mathcal{X}\subset\mathbb{R}^n$ and input constraints $\mathcal{U}\subset\mathbb{R}^m$ are box constraints, 
\small$$\mathcal{X}=\{\Vert D_xx-d_x\Vert_\infty\leq1\}\quad \mathcal{U}=\{\Vert D_uu-d_u\Vert_\infty\leq1\}$$ \normalsize
for some real, constant diagonal matrices $D_x, D_u$ and vectors $d_x,d_u$.
\end{assumption}

\begin{remark} The monotonicity property in assumption
~\ref{ass:flat_prop}\eqref{ass:flatmap_convex} is prevalent in literature (\cite{angeli2003monotone, yang2019sufficient}) for the system dynamics $f(\cdot)$. We require $\mathcal{F}(\cdot)$ to be monotonic instead of $f(\cdot)$.
\end{remark}
\begin{remark}
If the map $\mathcal{F}(\cdot)$ is linear-fractional (projective)
\small$$\mathcal{F}(\mathbf{y})=\frac{A\mathbf{y}+b}{c^\top \mathbf{y}+d},\quad c^\top \mathbf{y}+d>0,$$\normalsize
then Assumption \ref{ass:box} can be relaxed to requiring $\mathcal{X}$ and $\mathcal{U}$ to be any convex set. This follows because linear-fractional functions preserve convexity of sets \cite{boyd} by mapping the line between any two points in its domain space to a line between the images of the two points.
\end{remark}
\subsection{Convex Safe Set}\label{ConvOPSS}
Let $\mathbf{Y}^j_t=[y^j_t,\dots,y^j_{t+R}]$ be the lifted output at time $t$ and iteration $j$. Similarly, define the matrix $\mathbf{y}^j_t$ using $R$ outputs as

\begin{align}\label{eq:op_casc}
\mathbf{y}^j_t=[y^j_t,\dots,y^j_{t+R-1}]\in\mathbb{R}^{m\times R}  
\end{align}
Note that each $\mathbf{y}^j_t$ uniquely identifies a state $x^j_t$ via the map \eqref{eq:flatclass_x}, $\mathcal{F}_x(\mathbf{y}^j_t)=x^j_t$.
We define a successful iteration as one whose corresponding state trajectory converges to $x_F$ while simultaneously meeting state constraints $\mathcal{X}$ and input constraints $\mathcal{U}$. This implies that a successful iteration corresponds to a feasible trajectory of \eqref{eq:OP_inf}. 

For iteration $j$, define the \textit{Output Safe Set} as the set of $\mathbf{y}^j_t$s corresponding to trajectories from preceding successful iterations (denoted by $\mathcal{I}_j\subseteq\{0,1,\dots, j-1\}$), i.e.,\small
\begin{align}\label{eq:SSdef}
    \mathcal{SS}_{\mathbf{y}}^{j-1}=\bigcup\limits_{i\in\mathcal{I}_j}\bigcup\limits_{k=0}^{\infty}\big\{\mathbf{y}^i_k\big\}.
\end{align}\normalsize
Taking the convex hull of this set, we now define the \textit{Convex Output Safe Set}
\begin{align}\label{eq:CSdef}\small
    \mathcal{CS}_{\mathbf{y}}^{j-1}&=\textrm{conv}(\mathcal{SS}_{\mathbf{y}}^{j-1}).
\end{align}\normalsize
We now show that the Convex Safe Set  $\mathcal{CS}_{\mathbf{y}}^{j-1}$ is in fact, control invariant in the sense described in the following proposition and in Figure~\ref{fig:prop_exp}.
\begin{figure}[h]
    \centering
    \includegraphics[scale=0.42]{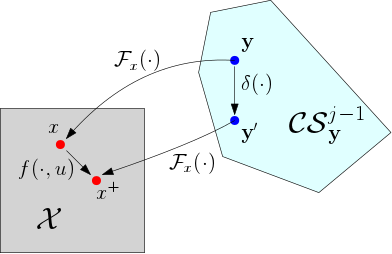}
    \caption{Illustration of the claim in Proposition \ref{prop:CS_CI}}
    \label{fig:prop_exp}
\end{figure}
\begin{proposition}\label{prop:CS_CI}
Under Assumptions~\ref{ass:flat_prop} and \ref{ass:box}, the set $\mathcal{CS}_{\mathbf{y}}^{j-1}$ defined in \eqref{eq:CSdef} is control invariant for the system dynamics \eqref{eq:sysdyn} subject to state constraints $\mathcal{X}$ and input constraints $\mathcal{U}$ in the following sense:
\begin{align}
\forall\mathbf{y}\in\mathcal{CS}_{\mathbf{y}}^{j-1}, \exists u \in \mathcal{U},\exists\mathbf{y}'\in\mathcal{CS}_{\mathbf{y}}^{j-1} :\mathcal{F}_x(\mathbf{y}')=f(x,u)\in\mathcal{X}\label{eq:imp_2_CI}
\end{align}
where $x=\mathcal{F}_x(\mathbf{y})\in\mathcal{X}$.
\end{proposition}

\begin{proof}
By definition of $\mathcal{CS}_{\mathbf{y}}^{j-1}$ we have for $\mathbf{y}\in\mathcal{CS}_{\mathbf{y}}^{j-1}$,\small
\begin{align}\label{eq:y_convCS}
&\mathbf{y}=\sum\limits_{i\in\mathcal{I}_{j-1}}\sum\limits_{t\geq 0} \lambda^i_t\mathbf{y}^i_t,\ \mathbf{y}^i_t\in\mathcal{SS}_{\mathbf{y}}^{j-1}\subset\mathbb{R}^{m\times R}\\
&\sum_{i\in\mathcal{I}_{j-1}}\sum_{k\geq 0}\lambda^i_k=1,\ \lambda^i_k\geq 0\nonumber
\end{align}\normalsize
By the definition of $\mathcal{SS}_{\mathbf{y}}^{j-1}$, each of the $\mathbf{y}^i_t$s in \eqref{eq:y_convCS} corresponds to a feasible state, meaning $\mathcal{F}_x(\mathbf{y}^i_t)\in\mathcal{X}$. Invoking Proposition~\ref{prop:box} then gives us, \small
\begin{equation}\label{eq:imp_1_CI}
    \mathcal{F}_x(\mathbf{y})=x\in\mathcal{X}
\end{equation}\normalsize
Again, the definition of $\mathcal{SS}_{\mathbf{y}}^{j-1}$ entails $\mathbf{y}^i_t\in\mathcal{SS}_{\mathbf{y}}^{j-1}\Rightarrow\mathbf{y}^i_{t+1}\in\mathcal{SS}_{\mathbf{y}}^{j-1}.$ We use the lifted output $\mathbf{Y}^j_t$ and map $\mathcal{F}_u(\cdot)$ to reconstruct the input applied in the $i$th iteration at time $t$ as $u^i_t=\mathcal{F}_u([y^i_{t},y^i_{t+1},\dots,y^i_{t+R}])=\mathcal{F}_u([y^i_{t},\mathbf{y}^i_{t+1}])$ and note that $u^i_t\in\mathcal{U}$ for all $i\in\mathcal{I}_{j-1}$ by the definition of the set $\mathcal{SS}_{\mathbf{y}}^{j-1}$. Consider the following control input
\small
\begin{align}\label{eq:safe_u}
    u&=\mathcal{F}_u(\sum\limits_{i\in\mathcal{I}_{j-1}}\sum\limits_{t\geq 0} \lambda^i_t[y^i_t,\mathbf{y}^i_{t+1}])\nonumber\\
    &=\mathcal{F}_u([\sum\limits_{i\in\mathcal{I}_{j-1}}\sum\limits_{t\geq 0} \lambda^i_t y^i_t,\sum\limits_{i\in\mathcal{I}_{j-1}}\sum\limits_{t\geq 0} \lambda^i_t \mathbf{y}^i_{t+1}])\nonumber\\
    &=\mathcal{F}_u([y,\mathbf{y}'])
\end{align}\normalsize
where $ y=\sum_{i\in\mathcal{I}_{j-1}}\sum_{t\geq 0} \lambda^i_t y^i_t\in\mathbb{R}^m$ and $\mathbf{y}'=\sum_{i\in\mathcal{I}_{j-1}}\sum_{t\geq 0} \lambda^i_t \mathbf{y}^i_{t+1}\in\mathbb{R}^{m\times R}$. Invoking proposition \ref{prop:box} again proves $u\in\mathcal{U}$. Also see that 
\begin{equation}\label{eq:time_shift_step}\small
    \mathbf{y}'=\sum_{i\in\mathcal{I}_{j-1}}\sum_{t\geq 0} \lambda^i_t \mathbf{y}^i_{t+1}\Rightarrow \mathbf{y}'\in\mathcal{CS}_{\mathbf{y}}^{j-1}.
\end{equation}\normalsize 
Let $u_2,\dots,u_{R-1}\in\mathbb{R}^m$ be the remaining inputs that generate $[y, \mathbf{y}']\in\mathbb{R}^{m\times R+1}$, i.e.,\small
\begin{align}\label{eq:flat_seq_gen}
    [y,\mathbf{y}']=&[h(x),h(f(x,u)),h(f^{(2)}(x,u,u_2)),\dots,\nonumber\\
    &h(f^{(R-1)}(x,u,\dots,u_{R-1}))]\in\mathbb{R}^{m\times R+1}
\end{align}\normalsize
where \small$f^{(k)}(x,u,\dots,u_k)=\underbrace{f(\dots(f}_{k\textrm{ times}}(x,u),\dots u_k).$\normalsize
Using the map \eqref{eq:flatclass_x} to construct state, we can write\small
\begin{align*}
    \mathcal{F}_x(\mathbf{y}')=&\mathcal{F}_x([h(f(x,u)),h(f^{(2)}(x,u,u_2)),\dots,\\&
    h(f^{(R-1)}(x,u,\dots,u_{R-1}))])\\
    =&f(x,u)
\end{align*}\normalsize
where the last equality is true because of the unique correspondence from \small$[y_t,\dots, y_{t+R-1}]=[h(x_t),\dots, h(f^{(R-1)}(x_t,u_t,\dots,u_{t+R-1}))]$\normalsize  to $x_t$ (Definition~\ref{def:diffFlat}). Finally, invoking proposition \ref{prop:box} using sequences $\mathbf{y}^i_{t+1}, \forall i\in\mathcal{I}_{j-1}$ gives us\small
\begin{equation}\label{eq:imp_3_CI}
f(x,u)\in\mathcal{X}.
\end{equation}\normalsize
\end{proof}

\begin{remark} 
Since $\mathcal{CS}_{\mathbf{y}}^{j-1}$ is not constructed on the state-space $\mathbb{X}$ directly, it not control invariant in the usual sense (\cite[Definition 10.9]{mpc}). But consider the forward-time shift operator $\delta(\cdot,\cdot)$ which defines dynamics on $\mathcal{CS}_{\mathbf{y}}^{j-1}$ as \small
\begin{align} \label{eq:time_shift}
   \mathbf{y}_{t+1}&=[y_{t+1},y_{t+2},\dots,y_{t+R}]\nonumber\\&=\delta( [y_{t}, y_{t+1},\dots,y_{t+R-1}], y_{t+R})\nonumber \\
   &=\delta(\mathbf{y}_t,y_{t+R})
\end{align}\normalsize
Notice from \eqref{eq:time_shift_step} in the proof of proposition \ref{prop:CS_CI} that the set $\mathcal{CS}_{\mathbf{y}}^{j-1}$ is control invariant on the space of output sequences with respect to the dynamics \eqref{eq:time_shift}, \small
\begin{align}
   \mathbf{y}_t\in\mathcal{CS}_{\mathbf{y}}^{j-1}\Rightarrow& \delta(\mathbf{y}_t,y_{t+R})=\mathbf{y}_{t+1}\in\mathcal{CS}_{\mathbf{y}}^{j-1}.\label{eq:pos_inv_delta}
\end{align}\normalsize
\end{remark}

The result of proposition \ref{prop:CS_CI} is  powerful; this allows us to consider the continuous set \eqref{eq:CSdef} instead of the discrete set \eqref{eq:SSdef} while still retaining the property of control invariance in the space of output sequences $\mathbf{y_t}$ with each pair $(\mathbf{y}_t,\mathbf{y}_{t+1})$ (equivalently, $\mathbf{Y}_t$) corresponding to state-input pairs within constraints. We use this continuous set for our MPC problem in Section~\ref{LMPC_Pol} to get a NLP instead of a MINLP.
\subsection{Convex Terminal Cost}\label{ConvOPTC}
Now we proceed to construct a terminal cost function which approximates the optimal cost-to-go from a state using lifted outputs from previous iterations. For some iteration $i$ and some time $t$, we define the cost-to-go for points in $\mathcal{SS}_{\mathbf{y}}^{j-1}$ as\small
\begin{equation}\label{eq:ctg}
    \mathcal{C}^i_t=\sum\limits_{k\geq t} C(\mathbf{y}^i_k)
\end{equation}\normalsize
 where the function $C(\cdot)$ is convex, continuous and satisfies\small
 \begin{align}\label{stage_cost}
         C(\mathbf{y}_F)=0,\ C(\mathbf{y})\succ 0 \ \forall \mathbf{y}\in\mathbb{R}^{m\times R}\backslash \{\mathbf{y}_F\}.
 \end{align}\normalsize
Observe that since each $\mathbf{y}\in\mathcal{SS}^{j-1}_{\mathbf{y}}$ corresponds to a unique $x$ via \eqref{eq:flatclass_x}, $C(\cdot)$ is an implicit function of state. We address the case of input costs in Section~\ref{input_costs}.
For iteration $j$, we use \eqref{eq:ctg} to construct the terminal cost on the convex safe set $\mathcal{CS}_{\mathbf{y}}^{j-1}$ using Barycentric interpolation (\cite{jones2010polytopic,rosolia2017learning}) with tuples $(\mathbf{y}^i_t, \mathcal{C}^i_t), \forall \mathbf{y}^i_t\in\mathcal{SS}_{\mathbf{y}}^{j-1}$.\small
\begin{equation}\label{eq:conv_cf}
\begin{aligned}
    Q^{j-1}(\mathbf{y})=\min\limits_{\substack{\lambda^i_k\in[0,1] \\ \forall i\in\mathcal{I}_{j-1}}} \quad & \sum\limits_{i\in\mathcal{I}_{j-1}}\sum\limits_{k\geq 0}\lambda^i_k\mathcal{C}^i_k\\[1ex]
		\text{s.t.} ~~\quad & \sum\limits_{i\in\mathcal{I}_{j-1}}\sum\limits_{k\geq 0}\lambda^i_k\mathbf{y}^i_k=\mathbf{y},\\
		& \sum\limits_{i\in\mathcal{I}_{j-1}}\sum\limits_{k\geq0}\lambda^i_k=1
\end{aligned}
\end{equation}\normalsize
For  any $\mathbf{y}\not\in\mathcal{CS}_{\mathbf{y}}^{j-1}$, we set $Q^{j-1}(\mathbf{y})=+\infty$. The following proposition identifies CLF-like characteristics of the function \eqref{eq:conv_cf} on the set $\mathcal{CS}_{\mathbf{y}}^{j-1}$ which we will use to show stability of the proposed controller in Section~\ref{sec:LMPC_analysis}.
\begin{proposition}\label{prop:Q_CLF}
The cost function $Q^{j-1}(\cdot)$ satisfies the following properties:
\begin{enumerate}
    \item $Q^{j-1}(\mathbf{y}_F)=0,\ Q^{j-1}(\mathbf{y})\succ 0\ \forall \mathbf{y}\in\mathcal{CS}_{\mathbf{y}}^{j-1}\backslash\{\mathbf{y}_F\}$ where $\mathbf{y}_F=[y_F,\dots,y_F]\in\mathbb{R}^{m\times R}$
    \item $Q^{j-1}(\mathbf{y}_{t+1})-Q^{j-1}(\mathbf{y}_t)\leq -C(\mathbf{y}_t),\ \forall \mathbf{y}_t\in\mathcal{CS}_{\mathbf{y}}^{j-1}$ where $\mathbf{y}_{t+1}=\delta(\mathbf{y}_t,y_{t+R})$ as in \eqref{eq:time_shift}.
\end{enumerate}
\end{proposition}
\begin{proof}
1) First note that $\mathbf{y}\in\mathcal{CS}_{\mathbf{y}}^{j-1}$ implies that the optimization problem implicit in the definition of $Q^{j-1}(\cdot)$ is feasible. Also see that since the feasible set is compact (closed subset of countable product of compact sets) and the objective is continuous (linear, in fact), a minimizer exists by Weierstrass' theorem for every $\mathbf{y}\in\mathcal{CS}_{\mathbf{y}}^{j-1}$. Thus for any $\mathbf{y}\in\mathcal{CS}_{\mathbf{y}}^{j-1}$, we can write 
$Q^{j-1}(\mathbf{y})= \sum\limits_{i\in\mathcal{I}_{j-1}}\sum\limits_{k\geq 0}\lambda^{\star i}_k\mathcal{C}^i_k$ where the $\lambda^{\star i}_k$s satisfy the constraints in \eqref{eq:conv_cf}. The definition of cost-to-go $\mathcal{C}^i_k$ in \eqref{eq:ctg} and positive definiteness of $C(\cdot)$ imply that $Q^{j-1}(\mathbf{y}) \succ 0 \ \forall \mathbf{y}\in\mathcal{CS}_{\mathbf{y}}^{j-1}\backslash\{\mathbf{y}_F\}$. We finish the proof for the first part by observing that $\mathbf{y}_F\in\mathcal{SS}_{\mathbf{y}}^{j-1}\subset\mathcal{CS}_{\mathbf{y}}^{j-1}$ by definition \eqref{eq:SSdef} and so $Q^{j-1}(\mathbf{y}_F)=0$.\\\\
2) For any $\mathbf{y}_t\in\mathcal{CS}_{\mathbf{y}}^{j-1}$, let $Q^{j-1}(\mathbf{y}_t)= \sum\limits_{i\in\mathcal{I}_{j-1}}\sum\limits_{k\geq 0}\lambda^{\star i}_k\mathcal{C}^i_k$ with $\lambda^{\star i}_k$ satisfying the constraints in \eqref{eq:conv_cf}. Observing the linearity of the forward-time shift operator $\delta(\cdot,\cdot)$, we have
\small
\begin{align*}
    \mathbf{y}_{t+1}&=\delta(\mathbf{y}_t,y_{t+R})\\
    &=\delta(\sum\limits_{i\in\mathcal{I}_{j-1}}\sum\limits_{k\geq 0}\lambda^{\star i}_k\mathbf{y}^i_k, \sum\limits_{i\in\mathcal{I}_{j-1}}\sum\limits_{k\geq 0}\lambda^{\star i}_k y^i_{k+R})\\
    &=\sum\limits_{i\in\mathcal{I}_{j-1}}\sum\limits_{k\geq 0}\lambda^{\star i}_k\delta(\mathbf{y}^i_k,y^i_{k+R})\\
    &=\sum\limits_{i\in\mathcal{I}_{j-1}}\sum\limits_{k\geq 0}\lambda^{\star i}_k\mathbf{y}^i_{k+1}.
\end{align*}\normalsize
Thus the same $\lambda^{\star i}_k$s are also feasible for \eqref{eq:conv_cf} at $\mathbf{y}_{t+1}$ and we have
\small
\begin{align*}
    Q^{j-1}(\mathbf{y}_{t+1})-Q^{j-1}(\mathbf{y}_t)&\leq \sum\limits_{i\in\mathcal{I}_{j-1}}\sum\limits_{k\geq 0}\lambda^{\star i}_k(\mathcal{C}^i_{k+1}-\mathcal{C}^i_{k})\\
    &= \sum\limits_{i\in\mathcal{I}_{j-1}}\sum\limits_{k\geq 0}\lambda^{\star i}_k(-C(\mathbf{y}^i_k))\\
    &\leq -C(\sum\limits_{i\in\mathcal{I}_{j-1}}\sum\limits_{k\geq 0}\lambda^{\star i}_k \mathbf{y}^i_k)\\
    &=-C(\mathbf{y}_t)
\end{align*}\normalsize
The second to last inequality comes from the convexity of $C(\cdot)$.
This completes the proof of the second part of the proposition.
\end{proof}
The above proposition shows that $Q^{j-1}(\cdot)$ is in fact a CLF for the dynamics $\mathbf{y}_{t+1}=\delta(\mathbf{y}_t, y_{t+R})$ with input $y_{t+R}$ on the convex output safe set $\mathcal{CS}_{\mathbf{y}}^{j-1}$. This is a critical property that we will use for our convergence analysis in section ~\ref{ssec:convergence}.
\subsection{LMPC Feedback Policy}\label{LMPC_Pol}
In this section, we show how to use constructions \eqref{eq:CSdef} and \eqref{eq:conv_cf} to design our LMPC policy. Before doing so, as in \cite{UgoTAC} we make the following assumption to initialise our recursive construction \eqref{eq:SS_def} of $\mathcal{SS}_{\mathbf{y}}^{j-1}$.

\begin{assumption}\label{ass:SSinit}
We are provided with an flat trajectory $\{y^0_t\}_{t\geq0}$ corresponding to a trajectory of system \eqref{eq:sysdyn} that is feasible for \eqref{eq:OP_inf} and converges to $x_F$.
\end{assumption}
Using Assumption~\ref{ass:SSinit}, the Output Safe Set \eqref{eq:SSdef} is initialised for $j=1$ as $\mathcal{SS}_{\mathbf{y}}^{0}=\bigcup\limits_{t=0}^{\infty}\big\{\mathbf{y}^0_t\big\}$ where $\mathbf{y}^0_t=[y^0_t,\dots, y^0_{t+R-1}]$.

 At iteration $j\geq 1$, we define the terminal cost on the space $\mathcal{CS}_{\mathbf{y}}^{j-1}$ as $Q^{j-1}(\cdot)$ and constrain the terminal state as $x_{N|t}=\mathcal{F}_x(\mathbf{y})$ for $\mathbf{y}\in\mathcal{CS}_{\mathbf{y}}^{j-1}$. The stage cost is set as $C(\cdot)$ which implicitly penalises only state. We address incorporating input costs in Section~\ref{input_costs}. Like the forward-shift operator \eqref{eq:time_shift}, we define the backward-time shift operator as \small
 \begin{align}\label{eq:b_time_shift}
    \mathbf{y}_{t}&=[y_{t},\dots,y_{t+R-1}]\nonumber\\&=\delta^{-}( [y_{t+1}, y_{t+1},\dots,y_{t+R}], y_{t})\nonumber\\
      &=\delta^{-}(\mathbf{y}_{t+1},y_{t})
 \end{align}\normalsize
 Employing these definitions, the LMPC optimization problem is given by
 \small
\begin{equation}\label{eq:OP_LMPC}
	\begin{aligned}
	J^j_{t\rightarrow t+N}(x^j_t)= \min\limits_{\boldsymbol{u}_t^j} \quad & Q^{j-1}(\mathbf{y}_{N|t})+\sum\limits_{k=0}^{N-1} C(\mathbf{y}_{k|t}) \\[1ex]
		\text{s.t.}  \quad & x_{k+1|t}=f(x_{k|t},u_{k|t}), \\
    & \mathbf{y}_{k|t} =\delta^{-}(\mathbf{y}_{k+1|t},h(x_{k|t}))\\
    &x_{k|t}\in\mathcal{X}, u_{k|t}\in\mathcal{U},\\
    &x_{N|t}=\mathcal{F}_x(\mathbf{y}_{N|t}),~ \mathbf{y}_{N|t}\in\mathcal{CS}_{\mathbf{y}}^{j-1},\\
		&  x_{0|t} = x^j_t, \\
		&\forall k\in \{0,\dots, N-1\}
	\end{aligned}
\end{equation}\normalsize
where the vector $\boldsymbol{u}_t^j = [u_{0|t},\ldots,u_{N-1|t}]$ are the decision variables whose optimal solution defines the LMPC control as
\begin{align}\label{eq:LMPC}
    u^j_t=\pi_{LMPC}(x^j_t)=u^*_0.
\end{align} 
Notice that the above control policy is well-defined for all state $x_0$ for which problem~\eqref{eq:OP_LMPC} is feasible. Thus, we define the region of attraction
\begin{equation}
    \mathcal{R}^j = \{ x\in \mathcal{X} | J^j_{t\rightarrow t+N}(x) < \infty \},
\end{equation}
which collects the states from which problem~\eqref{eq:OP_LMPC} is feasible. Next, we will show that for all states $x_0^j \in \mathcal{R}^{j}$ the closed-loop system is stable and it satisfies the state and input constraints.
\section{Properties of Proposed Strategy}\label{sec:LMPC_analysis}
\subsection{Recursive Feasiblity}
The next theorem and proof establishes the recursive feasibility of optimization problem \eqref{eq:OP_LMPC} for system \eqref{eq:sysdyn} in closed-loop with the LMPC policy \eqref{eq:LMPC}. We show this by leveraging the recursive definition of $\mathcal{CS}_{\mathbf{y}}^{j-1}$ and the result of Proposition~\ref{prop:CS_CI}. 
\begin{theorem}\label{thm:rf}
Given Assumptions \ref{ass:flat_prop} and \ref{ass:SSinit}, the optimization problem \eqref{eq:OP_LMPC} is recursively feasible for the system \eqref{eq:sysdyn} in closed-loop with the policy \eqref{eq:LMPC} for all iterations $j\geq 1$ and all time $t\geq 0$ with $x^j_0\in\mathcal{R}^{j}$.
\end{theorem}
\begin{proof}
For any iteration $j$, suppose that the problem \eqref{eq:OP_LMPC} is feasible at time $t$. Let the state-input trajectory corresponding to the optimal solution be \small\begin{equation}\label{eq:opt_sol_t}
    \{(x^\star_{0|t},u^\star_{0|t}), (x^\star_{1|t},u^\star_{1|t}),\dots, x^\star_{N|t}\}
\end{equation}\normalsize  with $x^\star_0=x^j_t$. Applying the LMPC control \eqref{eq:LMPC} $u^j_t=u^\star_{0|t}$ to system \eqref{eq:sysdyn} yields $x^j_{t+1}=x^\star_{1|t}$. Since \eqref{eq:opt_sol_t} is feasible for \eqref{eq:OP_LMPC}, we have
\small$$\mathbf{y}^\star_{N|t}\in\mathcal{CS}_{\mathbf{y}}^{j-1}, \quad x^\star_{N|t}=\mathcal{F}_x(\mathbf{y}_{N|t}^\star).$$\normalsize
From \eqref{eq:imp_1_CI} of proposition \ref{prop:CS_CI}, we have $x^\star_{N|t}\in\mathcal{X}$. From \eqref{eq:imp_2_CI}, we have $\mathcal{CS}_{\mathbf{y}}^{j-1}\ni \mathbf{y}'=\delta(\mathbf{y}_{N|t}^\star,y')$ and $\mathcal{U}\ni\tilde{u}=\mathcal{F}_u([y^\star_{N|t},\mathbf{y}'])$ such that $\mathcal{F}_x(\mathbf{y}')=f(x^\star_{N|t},\tilde{u})$. Finally from \eqref{eq:imp_3_CI}, we have $f(x^\star_{N|t},\tilde{u})=\tilde{x}\in\mathcal{X}$. Now consider the following state-input trajectory\small
\begin{equation}\label{eq:shift_sol_t}
    \{(x^\star_{1|t},u^\star_{1|t}),\dots, (x^\star_{N|t},\tilde{u}), \tilde{x}\}
\end{equation}\normalsize
This is feasible for the LMPC problem \eqref{eq:OP_LMPC} at time $t+1$. \\
We have shown that feasibility of the LMPC problem~\eqref{eq:OP_LMPC} at time $t$ implies feasibility of the LMPC problem~\eqref{eq:OP_LMPC} at time $t+1$. By assumption \ref{ass:SSinit}, we readily have $\cup^\infty_{t=0}\{\mathcal{F}_x(\mathbf{y}_t^0)\}\subseteq\mathcal{R}^{1}$. Also since $\mathcal{CS}_{\mathbf{y}}^{j-1}\subseteq\mathcal{CS}_{\mathbf{y}}^{j}$ by construction, we have $\mathcal{R}^j\neq\emptyset,~\forall j\geq 1$. So at $t=0$, the LMPC problem~\eqref{eq:OP_LMPC} is feasible with $x^j_0\in\mathcal{R}^j$ for all iterations $j\geq 1$. Thus induction on time $t$ proves the theorem.
\end{proof}
\subsection{Convergence}\label{ssec:convergence}
To establish convergence of the closed-loop trajectories of \eqref{eq:sysdyn} to the equilibrium $x_F$, we first present a lemma that shows that if $\lim_{t\rightarrow\infty}y_t=y_F$ then $\lim_{t\rightarrow\infty}x_t=x_F$. We use the result of this lemma to finally we show that $\lim_{t\rightarrow\infty}x^j_t=x_F$ in the proof of Theorem~\ref{thm:convergence}.
\begin{lemma}\label{lem:yconv_xconv}
If the trajectory of outputs $\{y_t\}_{t\geq0}$ for system \eqref{eq:sysdyn} converges to $y_F$ then the state trajectory $\{x_t\}_{t\geq 0}$ converges to $x_F$,
$$\lim_{t\rightarrow\infty}y_t=y_F\Rightarrow \lim_{t\rightarrow\infty}x_t=x_F$$
\end{lemma}
\begin{proof}
We sketch the proof of the statement, which proceeds in two steps. First, one can show that \small
$$\lim\limits_{t\rightarrow\infty}y_t=y_F\Rightarrow \lim\limits_{t\rightarrow\infty}\mathbf{y}_t=[y_F ,y_F,\dots, y_F]\in\mathbb{R}^{m\times R} $$\normalsize

Then the following implication \small
\begin{align}\label{eq:yflatconv_xconv}
    \lim\limits_{t\rightarrow\infty}\mathbf{y}_t=[y_F ,y_F,\dots, y_F]\Rightarrow\lim\limits_{t\rightarrow\infty}x_t=x_F
\end{align}\normalsize
follows from continuity of the flat map at $\mathbf{Y}_F=[y_F ,y_F,\dots, y_F]$ (Assumption \ref{ass:flat_prop}\eqref{ass:flat_cont}) and the fact that the image of the flat map \eqref{eq:flatclass_x} is unique (Definition~\ref{def:diffFlat}).
\end{proof}
\begin{theorem}\label{thm:convergence}
For any iteration $j\geq 1$, the system trajectory of \eqref{eq:sysdyn} in closed-loop with control $\eqref{eq:LMPC}$ converges to unforced equilibrium $x_F$,
$$\lim_{t\rightarrow\infty} x^j_t=x_F$$
with $x^j_0\in\mathcal{R}^{j}$.
\end{theorem}
\begin{proof}
At any iteration $j\geq 1$, the LMPC problem \eqref{eq:OP_LMPC} is time-invariant. This in turn implies that system \eqref{eq:sysdyn} in closed-loop with the LMPC control \eqref{eq:LMPC} is time-invariant and so we simply analyse the LMPC cost function $J^j_{0\rightarrow N}(\cdot)$ instead of $J^j_{t\rightarrow  t+N}(\cdot)$.
We adopt the same notation as the proof of theorem \ref{thm:rf}. Using the feasibility of \eqref{eq:shift_sol_t} for the LMPC problem at time $t+1$ and the fact that $J^{j}_{0\rightarrow N}(x^j_{t+1})$ is the optimal cost of problem \eqref{eq:OP_LMPC} at time $t+1$, we get\small
\begin{align}\label{eq:cost_decrease}
    J^j_{0\rightarrow N}(x^j_{t+1})&\leq Q^{j-1}(\mathbf{y}')+C(\mathbf{y}_{N|t}^\star)+\sum\limits_{k=1}^{N-1}C(\mathbf{y}^\star_{k|t})\nonumber\\
    &\leq Q^{j-1}(\mathbf{y}^\star_{N|t})+\sum\limits_{k=1}^{N-1}C(\mathbf{y}^\star_{k|t})\nonumber\\
    &=J^j_{0\rightarrow N}(x^j_t)-C(\mathbf{y}^j_t).
\end{align}\normalsize
Notice that the feasibility of the LMPC problem at all time steps is guaranteed by Theorem~\ref{thm:rf}.
Recursive feasibility and positive definiteness of $C(\cdot)$ imply that the sequence $\{J^j_{0\rightarrow N}(x^j_t)\}_{t\geq 0}$ is non-increasing. Moreover, positive definiteness of $Q^{j-1}(\cdot)$ further implies that sequence is lower bounded by $0$. Thus the sequence converges to some limit and taking limits on both sides of \eqref{eq:cost_decrease} gives\small
$$0\leq \lim_{t\rightarrow\infty}C(\mathbf{y}^j_t)\leq 0\Rightarrow \lim_{t\rightarrow\infty}C(\mathbf{y}^j_t)=0$$\normalsize
Continuity of $C(\cdot)$ and property \eqref{stage_cost} further imply that $\lim_{t\rightarrow\infty} \mathbf{y}^j_t=\mathbf{y}_F$. Finally using lemma \ref{lem:yconv_xconv} proves our claim,\small
$$\lim_{t\rightarrow\infty}\mathbf{y}^j_t=\mathbf{y}_F\Rightarrow \lim_{t\rightarrow\infty}x^j_t=x_F$$\normalsize
\end{proof}
\subsection{Performance Improvement}
We conclude our theoretical analysis of the proposed LMPC \eqref{eq:LMPC} with the following theorem. We state and prove that the closed-loop costs of system trajectories in closed-loop with the LMPC do not increase with iterations if the system starts from the same state, i.e., $x^j_0=x_S~\forall j\geq 0$.
\begin{theorem}\label{thm:cost_imp}
The cost of the trajectories of system \eqref{eq:sysdyn} in closed-loop with the LMPC \eqref{eq:LMPC} does not increase with iterations,
$$j_2>j_1\Rightarrow J^{j_2}_{0\rightarrow \infty}(x_S)\leq J^{j_1}_{0\rightarrow \infty}(x_S)$$
where $J^{j}_{0\rightarrow \infty}(x_S)=\mathcal{C}^j_0$.
\end{theorem}
\begin{proof}
The proof follows \cite{UgoTAC} closely. The cost of the trajectory in iteration $j-1$ is given by\small
\begin{align*}
    J^{j-1}_{0\rightarrow \infty}(x_S)&=\sum\limits_{t\geq 0} C(\mathbf{y}^{j-1}_t)\\
    &=\sum\limits_{t=0}^{N-1}C(\mathbf{y}^{j-1}_t)+\mathcal{C}^{j-1}_N\\
    &\geq\sum\limits_{t=0}^{N-1}C(\mathbf{y}^{j-1}_t)+Q^{j-1}(\mathbf{y}^{j-1}_N)\\
    &\geq J^j_{0\rightarrow N}(x_S)
\end{align*}\normalsize
The second to last inequality comes from the definition of $Q^{j-1}(\cdot)$ in \eqref{eq:conv_cf} while the last inequality comes from optimality of problem \eqref{eq:OP_LMPC} in the $j$th iteration starting from $x^j_0=x_S$.\\
Noting that $x_S=\mathcal{F}_x(\mathbf{y}^j_0)$, we use inequality \eqref{eq:cost_decrease} repeatedly to derive\small
\begin{align*}
    J^j_{0\rightarrow N}(x_S)\geq& C(\mathbf{y}^j_0)+J^j_{0\rightarrow N}(x^j_1)\\
    \geq& C(\mathbf{y}^j_0)+C(\mathbf{y}^j_1)+J^{j}_{0\rightarrow N}(x^j_2)\\
    \geq& \lim_{t\rightarrow\infty}( \sum\limits_{k=0}^{t-1}C(\mathbf{y}^j_k)+J^j_{0\rightarrow N}(x^j_t))
\end{align*}\normalsize
Observe that at $x_F$, the cost in \eqref{eq:OP_LMPC} $J^j_{0\rightarrow N}(x_F)=0$. Moreover, continuity of the dynamics \eqref{eq:sysdyn} and cost $C(\cdot)$ at $\mathbf{y}_F$ imply that $J^j_{0\rightarrow N}(\cdot)$ at $x_F$. Computing the above limit finally gives us
\small$$J^{j-1}_{0\rightarrow\infty}(x_S)\geq J^j_{0\rightarrow N}(x_S)\geq J^{j}_{0\rightarrow\infty}(x_S)$$\normalsize
The desired statement easily follows from above.   
\end{proof}
\subsection{Extension for Input Costs}\label{input_costs}
In this section, we show that our framework applies for costs of the form $c(x,u)$ as well. First, we augment system \eqref{eq:sysdyn} with input as a state with first-order compensator dynamics to get the \textit{augmented system}\small
\begin{align}\label{eq:aug_sysdyn}
    \tilde{x}_{t+1}=\begin{bmatrix}x_{t+1}\\u_{t+1}\end{bmatrix}=\tilde{f}(\tilde{x}_{t},z_t)=\begin{bmatrix}f(x_t,u_t)\\\alpha u_t+\beta z_t\end{bmatrix}
\end{align}\normalsize
with any scalars $\alpha,\beta\in\mathbb{R}$. Observe that if $\mathbf{Y}_t$ is a lifted output with $y_t=h(x_t), (x_t,u_t)=(\mathcal{F}_x(\mathbf{y}_t)$, $\mathcal{F}_u([\mathbf{y}_t,y_{t+R}])$ for \eqref{eq:sysdyn}, then $[\mathbf{Y}_t, y_{t+R+1}]$ is a lifted output for system \eqref{eq:aug_sysdyn} with $y_t=\tilde{h}(\tilde{x}_t)=h(x_t)$ and \small
\begin{align}
    \tilde{x}_t&=\tilde{\mathcal{F}}_x([\mathbf{y}_t,y_{t+R}])=\begin{bmatrix}\mathcal{F}_x(\mathbf{y}_t)\\\mathcal{F}_u([\mathbf{y}_t,y_{t+R}])\end{bmatrix}\label{eq:flatclass_x_aug}\\
    z_t&=\tilde{\mathcal{F}}_u([\mathbf{y}_t,y_{t+R},y_{t+R+1}])\nonumber\\
       &=\frac{1}{\beta}\left(\mathcal{F}_u([\delta(\mathbf{y}_t,y_{t+R}),y_{t+R+1}])-\alpha\mathcal{F}_u([\mathbf{y}_t,y_{t+R}])\right).\label{eq:flatclass_u_aug}
\end{align}\normalsize
So system \eqref{eq:aug_sysdyn} clearly satisfies Assumptions~\ref{ass:flat_prop}\eqref{ass:flat_class}, \ref{ass:flat_prop}\eqref{ass:flat_cont}. Moreover from \eqref{eq:flatclass_x_aug}, we see that $\tilde{\mathcal{F}}_x(\cdot)$ has the desired monotonicity properties in Assumption~\ref{ass:flat_prop}\eqref{ass:flatmap_convex}. If Assumption~\ref{ass:box} is in place, then $z_t$ isn't constrained explicitly and so $\tilde{\mathcal{F}}_u(\cdot)$ need not be monotonic as in Assumption~\ref{ass:flat_prop}\eqref{ass:flatmap_convex}.

For the augmented system \eqref{eq:aug_sysdyn}, we now define the \textit{Output Safe Set} as \small
\begin{align}\label{eq:SSdef_aug}
    \tilde{\mathcal{SS}}_{\mathbf{y}}^{j-1}=\bigcup\limits_{i\in\mathcal{I}_j}\bigcup\limits_{k=0}^{\infty}\big\{[\mathbf{y}^i_k,y^i_{k+R}]\big\}.
\end{align}\normalsize
Taking the convex hull of this set, we now define the \textit{Convex Output Safe Set}\small
\begin{align}\label{eq:CSdef_aug}
    \tilde{\mathcal{CS}}_{\mathbf{y}}^{j-1}&=\textrm{conv}(\tilde{\mathcal{SS}}_{\mathbf{y}}^{j-1}).
\end{align}\normalsize
The cost-to-go is defined on points in $\tilde{\mathcal{SS}}_{\mathbf{y}}^{j-1}$ as\small
\begin{align}\label{eq:ctg_aug}
    \tilde{\mathcal{C}}^i_t&=\sum\limits_{k\geq t} \tilde{C}(\tilde{\mathbf{y}}^i_k)
\end{align}\normalsize
where $\tilde{C}(\cdot)$ is a convex, continuous function which satisfies\small
\begin{align}\label{stage_cost_aug}
    \tilde{C}([\mathbf{y}_F,y_F])=0,\ C(\tilde{\mathbf{y}})\succ 0, \ \forall \tilde{\mathbf{y}}\in\mathbb{R}^{m\times R+1}\backslash \{[\mathbf{y}_F,y_F]\}.
\end{align}\normalsize
Unlike $C(\cdot)$, the function $\tilde{C}(\cdot)$ penalises both state and input implicitly via \eqref{eq:flatclass_x_aug}. The terminal cost is defined on $\tilde{\mathbf{y}}\in\tilde{\mathcal{CS}}_{\mathbf{y}}^{j-1}$ as\small
\begin{equation}\label{eq:conv_cf_aug}
\begin{aligned}
    \tilde{Q}^{j-1}(\tilde{\mathbf{y}})=\min\limits_{\substack{\lambda^i_k\in[0,1] \\ \forall i\in\mathcal{I}_{j-1}}} \quad & \sum\limits_{i\in\mathcal{I}_{j-1}}\sum\limits_{k\geq 0}\lambda^i_k\tilde{\mathcal{C}}^i_k\\[1ex]
		\text{s.t.} ~~\quad & \sum\limits_{i\in\mathcal{I}_{j-1}}\sum\limits_{k\geq 0}\lambda^i_k[\mathbf{y}^i_k,y^i_{k+R}]=\tilde{\mathbf{y}},\\
		& \sum\limits_{i\in\mathcal{I}_{j-1}}\sum\limits_{k\geq0}\lambda^i_k=1
\end{aligned}
\end{equation}\normalsize
With these revised definitions, all our results in Sections~\ref{sec:LMPCF} and \ref{sec:LMPC_analysis} still apply for system \eqref{eq:aug_sysdyn}.

\section{Numerical Examples}\label{sec:ex}
In this section, we first present results on three numerical experiments: (\ref{eg:PWA}) PWA system, (\ref{eg:uni}) Kinematic unicycle and (\ref{eg:DC}) Bilinear DC Motor. In each of these examples, we outline the system model, the associated lifted outputs and the components of the constrained optimal control problem. 
\subsection{PWA System}\label{eg:PWA}

We implement the proposed framework on the following Piecewise Affine (PWA) system:
\small
\begin{align}\label{eq:PWA_eg}
    x_{k+1}&=
    \begin{cases*}
    \begin{bmatrix}
    1& 0.2\\0&1
    \end{bmatrix}x_k+\begin{bmatrix}0\\1\end{bmatrix}u_k & if $[1\ 0]x_k\leq-2$\\
    \begin{bmatrix}
    1& 0.2\\0.5&1
    \end{bmatrix}x_k+\begin{bmatrix}0\\1\end{bmatrix}u_k+\begin{bmatrix}0\\1\end{bmatrix} & if $[1\ 0]x_k\geq-2$
    \end{cases*}
\end{align}
\normalsize
The lifted output and associated maps are given by 
\small
\begin{align}
    y_k&=\begin{bmatrix}1 & 0\end{bmatrix}x_k,\ \mathbf{Y}_k=[y_k,y_{k+1},y_{k+2}] \label{eq:op_eg}\\
    \mathcal{F}_x(y_k,y_{k+1})&=\begin{bmatrix}1& 0\\-5&5\end{bmatrix}\begin{bmatrix}y_k\\y_{k+1}\end{bmatrix}\label{eq:flat_x_eg}\\
    \mathcal{F}_u(y_k,y_{k+1},y_{k+2})&=\begin{cases*}
    [5\ -10\ 5]\begin{bmatrix}y_k\\y_{k+1}\\y_{k+2}\end{bmatrix} & if $y_k\leq-2$\\
    [4.5\ -10\ 5]\begin{bmatrix}y_k\\y_{k+1}\\y_{k+2}\end{bmatrix}-1 & if $y_k\geq-2$
    \end{cases*}\label{eq:flat_u_eg}
\end{align}
\normalsize
The map in\eqref{eq:flat_x_eg} is linear and the map in \eqref{eq:flat_u_eg} is monotonic and continuous. The state and input constraints are given by $\mathcal{X}=[-5,0]\times[0,6]$, $\mathcal{U}=[-10,2]$. In order to drive the system from the initial condition $x_S=[-5,0]$ to the origin, we minimize the sum of convex stage costs $C(\mathbf{y}_k)=5(y^2_k+y^2_{k+1})=x_k^\top\begin{bmatrix}10&1\\1&\frac{1}{5}\end{bmatrix}x_k$ over a prediction horizon $N=3$. The terminal set and terminal cost are constructed as \eqref{eq:CSdef} and \eqref{eq:conv_cf} respectively. The dynamics \eqref{eq:PWA_eg} are posed as constraints in the optimization problem \eqref{eq:OP_LMPC} using the Big-M formulation \cite{marcucci2019mixed} for a prediction horizon $N=3$.  The resulting problem is a MIQP with $3$ binary variables which we solve using GUROBI (one binary variable per time step $k$ to decide if $[1 0]x_k\leq-2$ or not along the prediction horizon $N=3$). Note that for the same problem, the formulation in \cite{UgoTAC} would have required $6+\vert \mathcal{SS}^{j-1}\vert$ binary variables where $\vert \mathcal{SS}^{j-1}\vert$ is the number of points in the Safe Set \eqref{eq:SS_def} at iteration $j$. At iterations $j=1,5,9$ of this example, the Safe Set \eqref{eq:SS_def} had $\vert \mathcal{SS}^{0}\vert=32$ points, $\vert \mathcal{SS}^{4}\vert=73$, $\vert \mathcal{SS}^{8}\vert=113$ points respectively.

We see that the proposed controller successfully steers the PWA system \eqref{eq:PWA_eg} to the origin (Figure \ref{fig:sim_eg}), while meeting state constraints and input constraints. The trajectory costs $J^j_{0\rightarrow\infty}=\sum_{k\geq0}C(\mathbf{y}^j_k)$ are non-increasing with iteration $j$ as is evident in Table~\ref{tab:Cost_PWA}.

\begin{figure}[h]
    \centering
    \includegraphics[width=1\linewidth]{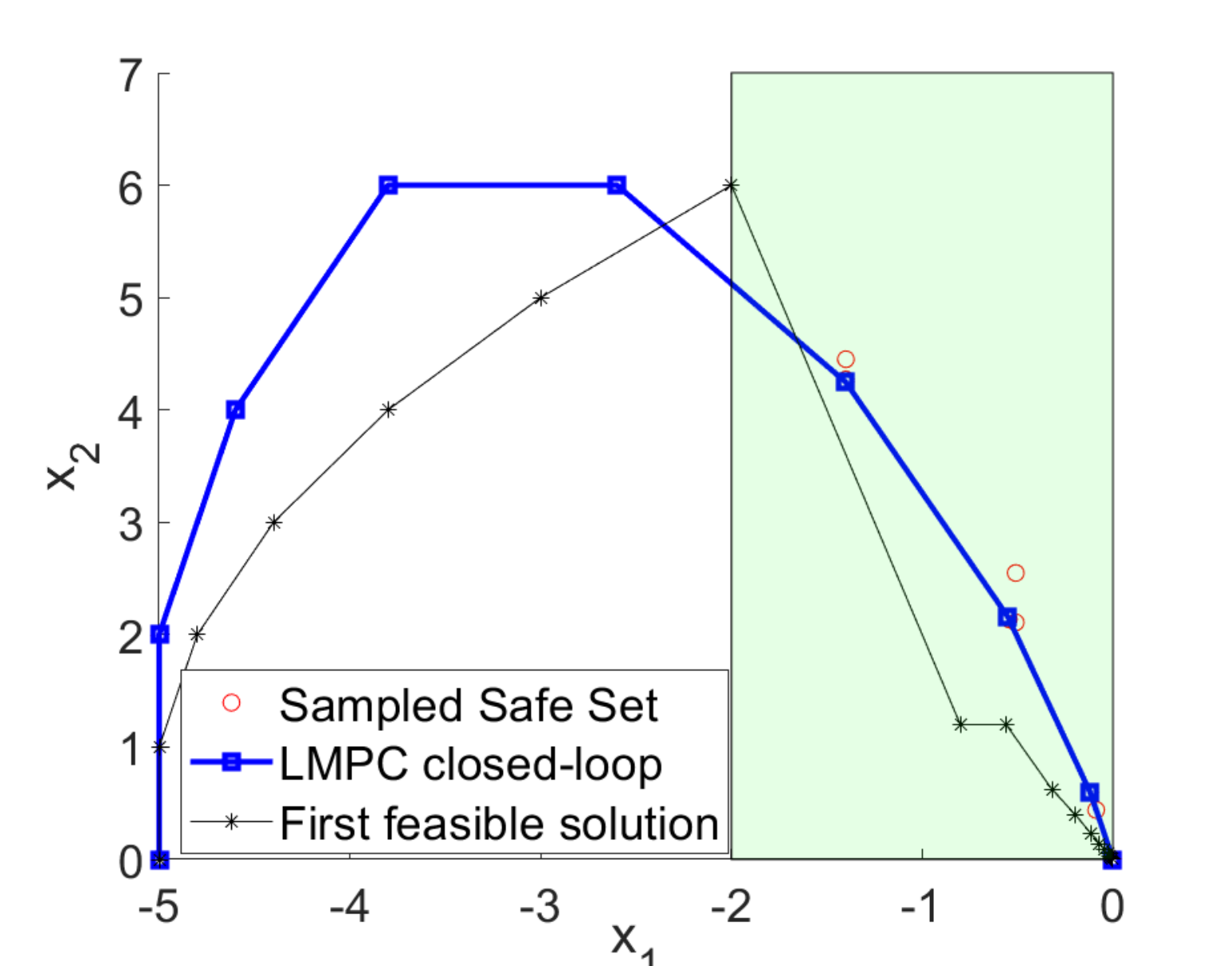}
    \caption{Closed loop realization for system \eqref{eq:PWA_eg} with LMPC. The coloured portion of the state space represents the mode $x_1\geq -2$. The final trajectory is indicated in blue.}
    \label{fig:sim_eg}
\end{figure}
\begin{table}[!h]
	\centering 
	\scriptsize
	\begin{tabular}{cccccc}
		\multicolumn{1}{c}{Iteration} &
		\multicolumn{1}{c}{0} & \multicolumn{1}{c}{1} & \multicolumn{1}{c}{2} & \multicolumn{1}{c}{3} & \multicolumn{1}{c}{4} 
		\\[-0.3cm] \\\hline \\[-0.2cm] 
		Cost$\times10^{-2}$   & 6.5831 & 5.0300 & 5.0209 & 5.0159 & 5.0156 \\
		[0.2cm]
		\multicolumn{1}{c}{Iteration} &
		\multicolumn{1}{c}{5} & \multicolumn{1}{c}{6} & \multicolumn{1}{c}{7} & \multicolumn{1}{c}{8} & \multicolumn{1}{c}{9} 
		\\[-0.3cm] \\\hline \\[-0.2cm] 
		Cost$\times10^{-2}$ & 5.0156 & 5.0156 & 5.0156 & 5.0156 & 5.0156  \\
		\end{tabular}
	\caption{Iteration Costs for system \eqref{eq:PWA_eg} in closed-loop with LMPC \eqref{eq:LMPC}.}
	\label{tab:Cost_PWA}
\end{table}

\subsection{Bilinear System- DC motor control} \label{eg:DC}
Consider the folllowing bilinear model of a DC motor
\small
\begin{align}\label{eq:DC_dyn}
    x_{k+1}=Ax_k+u_{1k}Bx_k+u_{2k}\begin{bmatrix}\frac{dt}{L_a}\\0\\0\end{bmatrix}
\end{align}
\normalsize
where  
\small
$$A=\begin{bmatrix}1-\frac{dt R_a}{L_a}& 0& 0\\0 &1& dt\\0&0&1-\frac{dt D}{J}\end{bmatrix}, B=\begin{bmatrix}0 & 0 & -\frac{K_y dt}{L_a}\\ 0&0&0\\\frac{K_y dt}{J} &0&0\end{bmatrix}$$
\normalsize
The state $x=[I\ \theta\ \omega]^\top$ comprises the armature current, motor angle and angular velocity, with inputs being field current $u_{1}$ and armature voltage $u_{2}$. The sampling period $dt$ is set to $0.01\ \text{s}$ and the other system parameters are taken from \cite{korda2018linear}. The lifted output and associated maps are given by 
\small
\begin{align}
    y_k&=\begin{bmatrix}1 & 0 & 0\\0 & 1 & 0\end{bmatrix}x_k=\begin{bmatrix}I_k\\\theta_k\end{bmatrix} \label{eq:op_eg_DC}\\
    \mathcal{F}_x(y_k,&y_{k+1})=\begin{bmatrix}1&0&0&0\\0&1&0&0\\0&\frac{-1}{dt}&0&\frac{1}{dt}\end{bmatrix}\begin{bmatrix}y_k\\y_{k+1}\end{bmatrix}=\begin{bmatrix}I_k\\\theta_k\\\omega_k\end{bmatrix}\label{eq:flat_x_eg_DC}\\
    \mathcal{F}_u(y_k,&y_{k+1},y_{k+2})=\begin{bmatrix}u_{1k}\\u_{2k}\end{bmatrix}\nonumber\\
    u_{1k} &=\frac{\begin{bmatrix}0 & \frac{1}{dt}\end{bmatrix}(y_{k+2}+(1-\frac{dt D}{J})y_{k}-(2-\frac{dt D}{J})y_{k+1})}{[\frac{K_y dt}{J}\ 0]y_k}\label{eq:flat_u_dc_eg1}\\
    u_{2k} &=\begin{bmatrix}\frac{L_a}{dt}&0\end{bmatrix}(y_{k+1}-(1-\frac{dt R_a}{L_a})y_k)\nonumber\\&+\begin{bmatrix}0& K_yu_{1k}\end{bmatrix}(y_{k+2}-y_{k+1})\label{eq:flat_u_dc_eg2}
\end{align}
\normalsize
The map in \eqref{eq:flat_x_eg_DC} is linear and the map in \eqref{eq:flat_u_dc_eg1} is continuous and linear-fractional for $I_k>0$. The state and input constraints are given by $\mathcal{X}=[0,5]\times\mathbb{R}\times[-10,10]$, $\mathcal{U}=[-5,5]\times\mathbb{R}$. For tracking the set-point $\omega=6.0\ \text{rad/s}$, we minimize the sum of convex stage costs \small$c(x_k,u_k)=20(\omega_k-6)^2+20(I_k-I^*)^2+(u_{1k}-u_1^*)^2$\normalsize where $I^*, u^*_1$ are the equilibrium armature and field current obtained from $\mathcal{SS}^0_{\mathbf{y}}$. Using the augmented system formulation in Section~\ref{input_costs}, we use stage cost \small$\tilde{C}(\tilde{\mathbf{y}}_k)=20([0\ 0\ 1]\mathcal{F}_x(\mathbf{y}_k)-6)^2+20([1\ 0\ 0]\mathcal{F}_x(\mathbf{y}_k)-I^*)^2+([1\ 0]\mathcal{F}_u(\tilde{\mathbf{y}}_k)-u^*_1)^2$\normalsize in \eqref{eq:OP_LMPC} with prediction horizon $N=5$. It is convex because \eqref{eq:flat_x_eg_DC} is linear and \eqref{eq:flat_u_dc_eg1} is linear-fractional (and hence, monotonic \cite{boyd}). The optimization problem \eqref{eq:OP_LMPC} with the terminal set and terminal cost constructed as in \eqref{eq:CSdef} and \eqref{eq:conv_cf} respectively is a NLP, solved using fmincon in MATLAB.

In Figure~\ref{fig:sim_eg_DC}, we see that closed-loop system \eqref{eq:DC_dyn} successfully tracks the desired set-point while meeting state constraints and input constraints (Figure~\ref{fig:input_DC}). The trajectory costs $J^j_{0\rightarrow\infty}=\sum_{k\geq0}\tilde{C}(\tilde{\mathbf{y}}^j_k)$ are non-increasing with iteration $j$ as is evident in Table~ \ref{tab:Cost_DC}.
\begin{figure}[!h]
    \centering
    \includegraphics[width=1\linewidth]{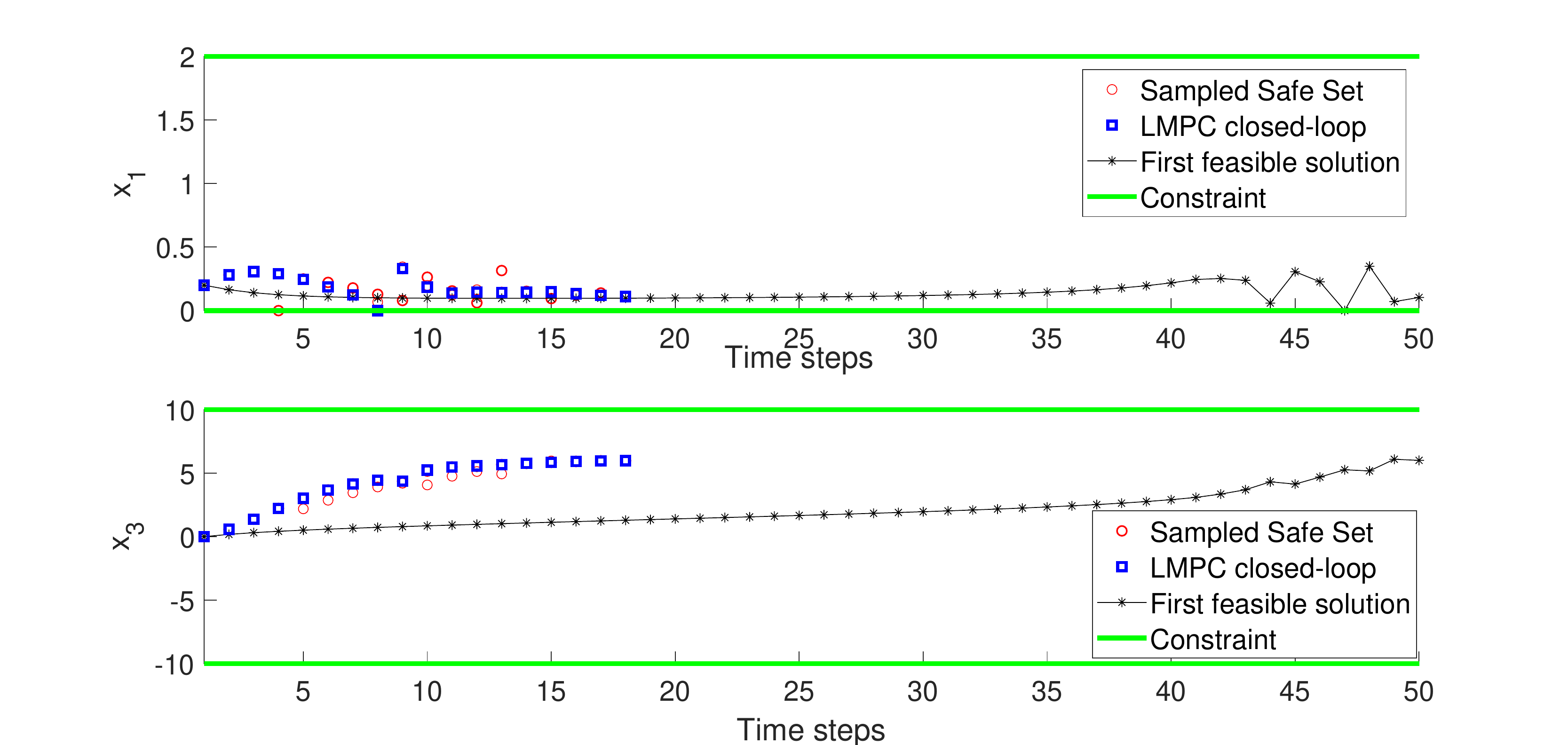}
    \caption{Armature current and angular velocity trajectories in closed-loop with LMPC \eqref{eq:LMPC}. The final trajectory is indicated in blue.}
    \label{fig:sim_eg_DC}
\end{figure}
\begin{figure}[!h]
    \centering
    \includegraphics[width=1\linewidth]{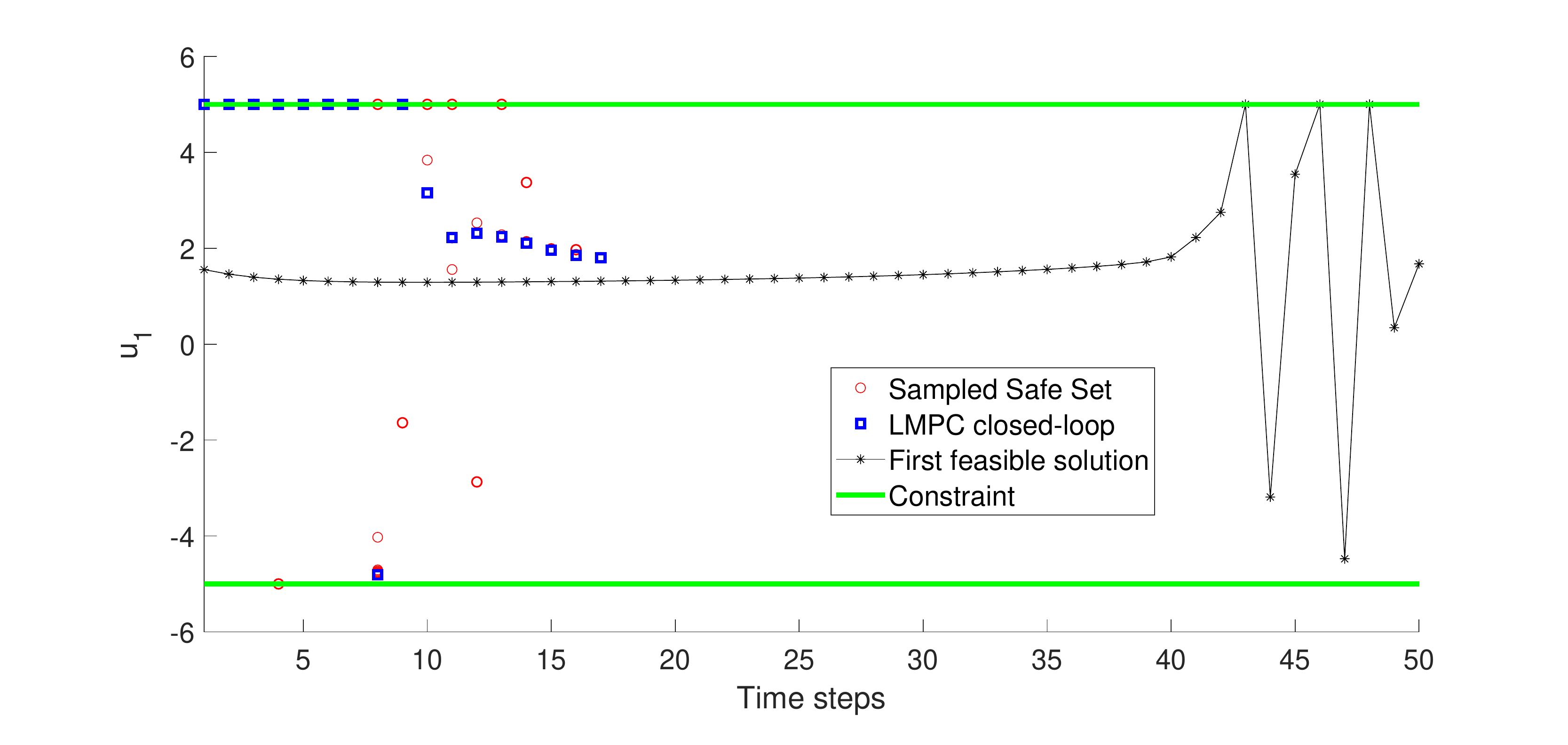}
    \caption{Field current  applied over time across iterations}
    \label{fig:input_DC}
\end{figure}
\begin{table}[!h]
	\centering 
	\scriptsize
	\begin{tabular}{cccccc}
		\multicolumn{1}{c}{Iteration} &
		\multicolumn{1}{c}{0} & \multicolumn{1}{c}{1} & \multicolumn{1}{c}{2} & \multicolumn{1}{c}{3} & \multicolumn{1}{c}{4} 
		\\[-0.3cm] \\\hline \\[-0.2cm] 
		Cost$\times10^{-4}$   & 1.7801 & 0.3138 & 0.2653 & 0.2653 & 0.2643 \\
		[0.2cm]
		\multicolumn{1}{c}{Iteration} &
		\multicolumn{1}{c}{5} & \multicolumn{1}{c}{6} & \multicolumn{1}{c}{7} & \multicolumn{1}{c}{8} & \multicolumn{1}{c}{9} 
		\\[-0.3cm] \\\hline \\[-0.2cm] 
		Cost$\times10^{-4}$ & 0.2643 & 0.2643 & 0.2643 & 0.2643 & 0.2643  \\
		\end{tabular}
	\caption{Iteration Costs for system \eqref{eq:DC_dyn} in closed-loop with LMPC \eqref{eq:LMPC}.}
	\label{tab:Cost_DC}
\end{table}
\subsection{Kinematic Unicycle}\label{eg:uni}
Consider the following kinematic unicycle model with state $x_k=[X_k\ Y_k\ \theta_k]^\top$, controls $u_k=[v_k\ w_k]^\top$ and discretization step $dt=0.1$ s\small
\begin{align}\label{eq:uni_kin}
    \begin{bmatrix}X_{k+1}\\Y_{k+1}\\\theta_{k+1}\end{bmatrix}=\begin{bmatrix}X_k\\Y_k\\ \theta_k\end{bmatrix} +dt\begin{bmatrix} v_k\cos(\theta_k)\\v_k\sin(\theta_k)\\w_k\end{bmatrix}
\end{align}\normalsize
The lifted output and associated maps are given by 
\small
\begin{align}
    y_k&=\begin{bmatrix}X_k\\Y_k\end{bmatrix},\ \mathbf{Y}_k=[y_k,y_{k+1},y_{k+2}]\label{eq:op_eg}
\end{align}
\begin{align}
    \mathcal{F}_x(y_k,y_{k+1})&=\begin{bmatrix}y_k\\\tan^{-1}\left(\dfrac{\begin{bmatrix}0&1\end{bmatrix}(y_{k+1}-y_{k})}{\begin{bmatrix}1&0\end{bmatrix}(y_{k+1}-y_{k})}\right)\end{bmatrix}\label{eq:flat_x_eg_uni}\\
    \mathcal{F}_u(y_k,y_{k+1},y_{k+2})&=\begin{bmatrix}v_k\\w_k\end{bmatrix}\nonumber
    \end{align}
\begin{align}
    v_k&=\frac{1}{dt}\Vert y_{k+1}-y_{k}\Vert_2\label{eq:flat_u_eg_uni}\\
    w_k&=\begin{bmatrix}0 & 0 & \frac{1}{dt}\end{bmatrix}\big(\mathcal{F}_x(y_{k+1},y_{k+2})-\mathcal{F}_x(y_{k},y_{k+1})\big)
\end{align}
\normalsize
 From \eqref{eq:flat_x_eg_uni}, we see that $\mathcal{F}_x(\cdot)$ is linear in its first two components and monotonic in the third component (composition of monotonic and quasilinear map \cite{boyd}). For speed input $v_k$, \eqref{eq:flat_u_eg_uni} is quasiconvex and doesn't require quasiconcavity (speed is always positive). The state and input constraints are given by $\mathcal{X}=\{(x,y)\in\mathbb{R}^2\vert(x\geq 0)\wedge (y\leq 10) \wedge (x-y\leq 2)\}\times[-\frac{\pi}{2},\frac{\pi}{2}]$, $\mathcal{U}=[0,5]\times\mathbb{R}$. To steer the unicycle to the position $(5,10)$, we minimize the sum of convex stage costs \small$c(x_k,u_k)=20(X_k-5)^2+20(Y_k-10)^2+v_k^2$\normalsize. Using the augmented system formulation in Section~\ref{input_costs}, we use stage cost \small$\tilde{C}(\tilde{\mathbf{y}}_k)=20\Vert y_k-[5\ 10]^\top\Vert_2^2+\Vert [1\ 0]\mathcal{F}_u(\tilde{\mathbf{y}}_k)\Vert_2^2$\normalsize 
 over a prediction horizon $N=5$ in \eqref{eq:OP_LMPC}. This cost is convex in $\tilde{\mathbf{y}}$ because of quasiconvexity of \eqref{eq:flat_u_eg_uni}. The optimization problem \eqref{eq:OP_LMPC} with the terminal set and terminal cost constructed as in \eqref{eq:CSdef} and \eqref{eq:conv_cf} respectively is a NLP, solved using fmincon in MATLAB.

We see that the proposed controller successfully steers the unicycle \eqref{eq:uni_kin} to the position $(5,10)$ (Figure \ref{fig:sim_eg_uni}), while meeting state constraints and input constraints (Figure~\ref{fig:input_uni}). The trajectory costs $J^j_{0\rightarrow\infty}=\sum_{k\geq0}\tilde{C}(\tilde{\mathbf{y}}^j_k)$ are non-increasing with iteration $j$ as is evident in Table~ \ref{tab:Cost_uni}.
\begin{figure}[!h]
    \centering
    \includegraphics[width=1\linewidth]{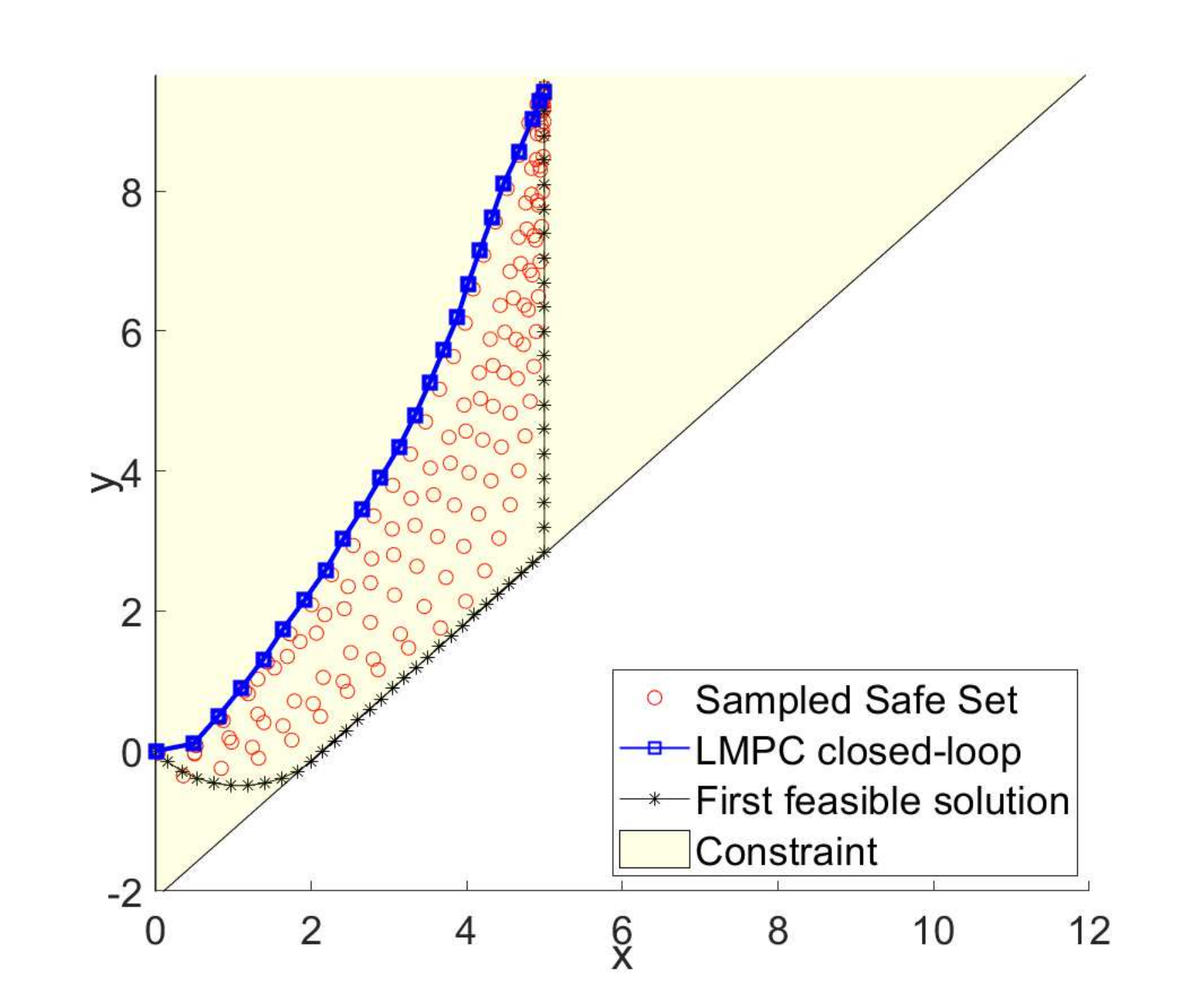}
    \caption{Unicycle \eqref{eq:uni_kin} trajectories  in closed-loop with LMPC. The final trajectory is indicated in blue.}
    \label{fig:sim_eg_uni}
\end{figure}

\begin{figure}[!h]
    \centering
    \includegraphics[width=1\linewidth]{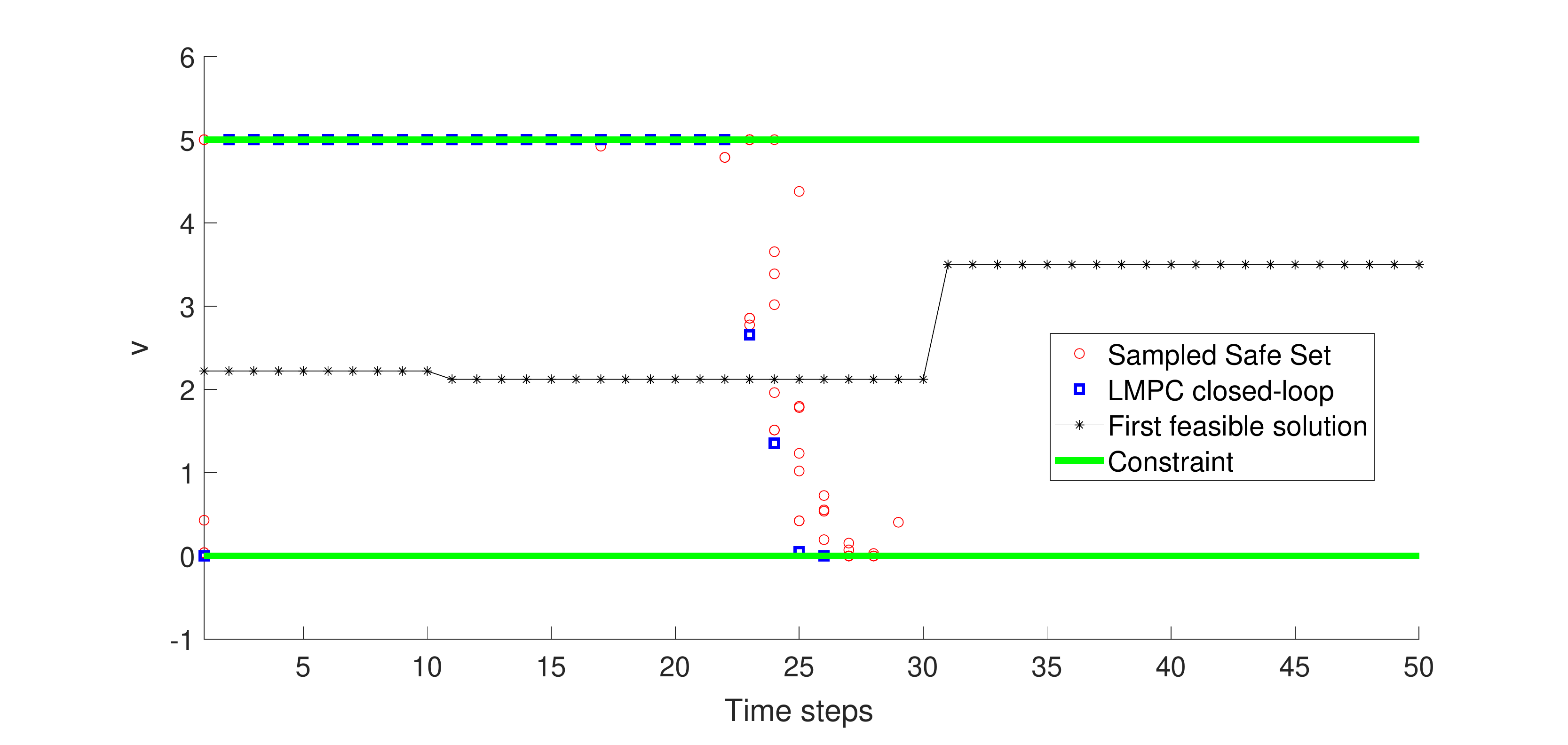}
    \caption{Speed profile over time across iterations. The final trajectory (blue) reaches $(5,10)$ the fastest.}
    \label{fig:input_uni}
\end{figure}

\begin{table}[!h]
	\centering 
	\scriptsize
	\begin{tabular}{cccccc}
		\multicolumn{1}{c}{Iteration} &
		\multicolumn{1}{c}{0} & \multicolumn{1}{c}{1} & \multicolumn{1}{c}{2} & \multicolumn{1}{c}{3} & \multicolumn{1}{c}{4} 
		\\[-0.3cm] \\\hline \\[-0.2cm] 
		Cost$\times10^{-4}$   & 6.4931 & 2.8536 & 2.8088 & 2.7301 & 2.6944 \\
		[0.2cm]
		\multicolumn{1}{c}{Iteration} &
		\multicolumn{1}{c}{5} & \multicolumn{1}{c}{6} & \multicolumn{1}{c}{7} & \multicolumn{1}{c}{8} & \multicolumn{1}{c}{9} 
		\\[-0.3cm] \\\hline \\[-0.2cm] 
		Cost$\times10^{-4}$ & 2.6624 & 2.5695 & 2.4971 & 2.4515 & 2.4343  \\
		\end{tabular}
	\caption{Iteration Costs for system \eqref{eq:uni_kin} in closed-loop with LMPC \eqref{eq:LMPC}.}
	\label{tab:Cost_uni}
\end{table}

\section{Conclusion}
We have proposed a revised formulation of LMPC for systems with lifted outputs performing iterative tasks. We showed that with certain properties of these outputs, we can solve infinite-horizon constrained optimal control problems by planning in the space of lifted outputs. A recursively-feasible, stabilizing LMPC strategy was proposed via the construction of a convex control invariant set and an accompanying CLF on this space of lifted outputs using historical data. We leave incorporation of model uncertainty into our formulation as future work.
\section*{Acknowledgement}
We would like to thank Koushil Sreenath for helpful discussions. This work was also sponsored by the Office of Naval
Research. The views and conclusions contained herein are those of the authors and should not be
interpreted as necessarily representing the official policies or endorsements, either expressed or
implied, of the Office of Naval Research or the US government.

\bibliographystyle{ieeetr}
\bibliography{root.bib}
\section{Appendix}
\subsection{Proof of Proposition \ref{prop:box}}\label{proof:box}
Before proving the statement, we first prove the following auxiliary property that is granted by Assumption \ref{ass:flat_prop}\eqref{ass:flatmap_convex}:
\small
\begin{align*}
  \mathcal{F}^i(\mathbf{y})\in[\min\limits_{k=1,\dots, p}\mathcal{F}^i(\mathbf{y}^k),&\max\limits_{k=1,\dots, p}\mathcal{F}^i(\mathbf{y}^k)]\\
  &\forall i=1,\dots,n+m  
\end{align*}\normalsize

We proceed using induction on $p$, the number of points in the set. For $p=2$, the property follows trivially by the definition of monotonicity of $\mathcal{F}^i(\cdot)$ along the line joining $\mathbf{y}^1$ and $\mathbf{y}^2$. Suppose the property is true for $p-1$, i.e.,\small
\begin{align*}
\mathcal{F}^i(\mathbf{y})\in[\min\limits_{k=1,\dots, p-1}\mathcal{F}^i(\mathbf{y}^k),&\max\limits_{k=1,\dots, p-1}\mathcal{F}^i(\mathbf{y}^k)]\\
&\forall i=1,\dots,n+m
\end{align*}\normalsize

Adding an additional point in the set, and writing $\textrm{conv}(\{\mathbf{y}^1,\dots,\mathbf{y}^p\})\ni\mathbf{y}=\lambda\mathbf{y}^p+(1-\lambda)\mathbf{y}'$ for some $\lambda\in[0,1]$ and $\mathbf{y}' \in\textrm{conv}(\{\mathbf{y}^1,\dots,\mathbf{y}^{p-1}\})$. Using the property for $p=2$, we have\small
\begin{align*}
\mathcal{F}^i(\mathbf{y})\in[\min(\mathcal{F}^i(\mathbf{y}^p),\mathcal{F}^i(\mathbf{y}')),&\max(\mathcal{F}^i(\mathbf{y}^p),\mathcal{F}^i(\mathbf{y}'))]\\
& \forall i=1,\dots,n+m
\end{align*}\normalsize
Using the truth of property for $p-1$ we have for all $i=1,\dots,n+m$
\small
\begin{align*}
\min\limits_{k=1,\dots, p}\mathcal{F}^i(\mathbf{y}^k)\leq\min(\mathcal{F}^i(\mathbf{y}^p),\mathcal{F}^i(\mathbf{y}'))\\ 
\max(\mathcal{F}^i(\mathbf{y}^p),\mathcal{F}^i(\mathbf{y}')) \leq \max\limits_{k=1,\dots, p}\mathcal{F}^i(\mathbf{y}^k)\\
\Rightarrow \mathcal{F}^i(\mathbf{y})\in[\min\limits_{k=1,\dots, p}\mathcal{F}^i(\mathbf{y}^k),\max\limits_{k=1,\dots, p}\mathcal{F}^i(\mathbf{y}^k)]
\end{align*}\normalsize
The property thus holds true for $p$ as well and induction helps us conclude that this holds for any $p\geq 1$. Since each $\mathcal{F}(\mathbf{y}^k)\in\mathcal{X}\times\mathcal{U}$ and the sets $\mathcal{X}\times\mathcal{U}$ are defined by box constraints, this implies that
$$\mathcal{F}(\mathbf{y})\in\mathcal{X}\times\mathcal{U}\quad\forall\mathbf{y}\in\textrm{conv}(\{\mathbf{y}^1,\dots,\mathbf{y}^p\})$$
\hfill$\blacksquare$



\end{document}